\documentclass{article}
\usepackage{amsmath}
\usepackage{amssymb}
\usepackage{geometry}
\usepackage{graphicx}

\newtheorem{theorem}{Theorem}[section]

\newtheorem{corollary}[theorem]{Corollary}

\newtheorem{definition}[theorem]{Definition}

\newtheorem{proposition}[theorem]{Proposition}

\newenvironment{proof}[1][Proof]{\noindent\textbf{#1.} }{\ \rule{0.5em}{0.5em}}
\input{tcilatex}
\begin{document}

\title{The Orbit Group of a Quandle}
\author{Sriram Nagaraj \\
Department of Electrical Engineering\\
The University of Texas, at Dallas\\
Richardson, TX 75083\\
sriram@student.utdallas.edu}
\date{}
\maketitle

\begin{abstract}
We define the notion of the orbit group of a quandle at a point, denoted by $%
O(Q,a)$. We calculate the orbit groups for some elementary quandles and show
that $O(Q,a)$ "counts" the number of orbits of a quandle.

\noindent 
\it{Keywords}: \textup{knot, link, quandle}%
\bigskip

\noindent 
\textup{Mathematics Subject Classification 2000: Primary 57M27; Secondary 55M99}%
\end{abstract}

\section{Introduction}

Quandles are algebraic structures that arise in the theory of knots/link as
their defining axioms correspond to the Reidemeister moves and have been
studied extensively in \cite{J-1}, \cite{J-2}, etc. We will briefly review
basic definitions and facts regarding quandles.

\begin{definition}
A quandle $(Q,\ast )$ is a set $Q$ with a binary operation $\ast :Q\times
Q\longrightarrow Q$ that satisfies the following axioms$:$

$\mathbf{(i)}$ For all $q\in \ Q,$ $q\ast \ q=q$

$\mathbf{(ii)}$ $Q$ is a self distributive set with $\ast $\ as an
operation, i.e. for all $p,q,r\in \ Q,$

\begin{equation*}
(p\ast q)\ast r=(p\ast r)\ast (q\ast r)
\end{equation*}

and finally,

$\mathbf{(iii)}$ For each $p,q\in \ Q$ there is a unique $r\in \ Q$ such
that $p=r\ast q$
\end{definition}

In case $Q$ is commutative, we call $Q$ a commutative quandle\footnote{%
Note that commutative and Abelian are not synonymous in the theory of
quandles (or racks). Abelian quandles satisfy $(a\ast b)\ast (c\ast
d)=(a\ast c)\ast (b\ast d)$ for all $a,b,c,d$ in the quandle (resp. rack)
whereas commutative quandles satisfy $a\ast b=b\ast a$ for all $a,b$ in the
quandle (resp. rack).}.\textit{\ }If $Q$ satisfies axiom $\mathbf{(ii)}$ and 
$\mathbf{(iii)}$ above, it is called a \textbf{rack}\textit{.}

Standard examples of quandles are the following:

\noindent \textbf{Example 1:} (Trivial Quandle) Any set $X$ with the
relation $x\ast y=x$ for all $x,y\in X$ is a quandle called the trivial
quandle. If $|X|=n$, $X$ is called the trivial quandle of rank $n$ and
denoted by $T_{n}$\medskip

\noindent \textbf{Example 2:} (Dihedral Quandles) For integers $i,j\in 
\mathbb{Z}
/n%
\mathbb{Z}
$, the product $i\ast j=2j-i$ (mod $n$) gives the group $%
\mathbb{Z}
/n%
\mathbb{Z}
$ the structure of a quandle called the dihedral quandle written as $D_{n}$%
.\medskip

Other examples can be found in \cite{NH}, \cite{NY} etc.

Consider the map $\ast ^{a}:Q\rightarrow Q$ defined by $\ast ^{a}(b)=b\ast a$%
. This is clearly an automorphism of $Q$ by properties $\mathbf{(ii)}$ and $%
\mathbf{(iii)}$\textbf{\ }defining a quandle. Moreover, by defining $\ast
^{a}\circledast \ast ^{b}=\ast ^{a\ast b}$, we can see that 
\begin{equation*}
\ast ^{a}\circledast \ast ^{a}=\ast ^{a}
\end{equation*}%
\begin{equation*}
(\ast ^{a}\circledast \ast ^{b})\circledast \ast ^{c}=\ast ^{(a\ast b)\ast
c}=\ast ^{(a\ast c)\ast (b\ast c)}=(\ast ^{a\ast c})\circledast (\ast
^{b\ast c})=(\ast ^{a}\circledast \ast ^{c})\circledast (\ast
^{b}\circledast \ast ^{c})
\end{equation*}%
and finally that given $\ast ^{a},\ast ^{b}$ there is a unique $\ast ^{c}$
(with $c$ such that $a=c\ast b$) such that $\ast ^{a}=\ast ^{c}\circledast
\ast ^{b}$ so that%
\begin{equation*}
Q^{\ast }=\{\ast ^{a}|\text{ }a\in Q\}
\end{equation*}

is a quandle in its own right.

By $\mathbf{(iii)}$ of the definition of a quandle, for all $a,b\in Q$ there
exists an \textbf{inverse}, $\ast ^{-1}$of $\ast $ in the sense that $(a\ast
b)\ast ^{-1}b=a$ and $(a\ast ^{-1}b)\ast b=a$. Thus, $Q$ with the operation $%
\ast ^{-1}$ is also a quandle called the inverse quandle of $Q$ and denoted
by $Q^{-1}$. Also, we may define the map $\ast ^{-a}:Q\rightarrow Q$ as $%
\ast ^{-a}(b)=b\ast ^{-1}a$. Moreover, $\ast ^{a}\circ \ast ^{-a}=id$, the
identity map.

The motivation for this paper is to quantitatively analyze the results
obtained in \cite{NY} regarding the orbit decomposition of finite quandles.
Our paper is organized as follows. Section $2$ is devoted to introducing the
idea of paths in a quandle and also connectivity of quandles. The remainder
of the paper will concentrate on defining and analyzing the orbit group $O(Q)
$ of a quandle $Q$. The orbit group "counts" in a certain sense, the number
of orbits in a finite quandle. We compute the orbit group of the trivial
quandle and dihedral quandles and show that they isomorphic to $k$ direct
sums of $%
\mathbb{Z}
/2%
\mathbb{Z}
$ where $k$ is the number of orbits. These computations motivate our main
result on the orbit group of a generic quandle with $k$ orbits.

\section{Paths and Connectivity}

We begin this section with the following definition:

\begin{definition}
Let $a\in Q$. The orbit $O_{a}$ of $a$ is defined as$:$%
\begin{equation*}
O_{a}=\{(...(((a\diamond a_{1})\diamond a_{2})\diamond a_{3})...)\diamond 
\text{ }a_{s}|\text{ }a_{i}\in Q\text{ for }i=1,2...,s\}\text{ and }\diamond 
\text{ }=\{\ast ^{\pm }\}
\end{equation*}%
$Q$ is \textbf{connected} if there is only one orbit i.e. for any $a,b\in Q$%
, there exist $a_{1},a_{2},...,a_{k}\in Q$ such that $b=(...((a\diamond
a_{1})\diamond a_{2})...)\diamond a_{k}$. $($See $\cite{NY}$ for the notion
of strongly-connected or Latin quandles$)$
\end{definition}

In \cite{NY} and \cite{GA} it was shown that any quandle $Q$ may be written
as the disjoint union of indecomposable maximal subquandles, and that
"indecomposable maximal subquandle" and "orbit of an element" are
synonymous. In other words, given a quandle $Q$, we can write $%
Q=\dbigsqcup\limits_{\substack{ \lbrack x]  \\ x\in Q}}O_{x}$, where the
union is over all equivalence classes of $x$ under the identification $x\sim
y$ if $x$ and $y$ are in the same orbit.

Thus, elements in $Q$ are connected iff they are in the same orbit. To
continue with our discussion, we need to define the notion of a path
connecting elements $a,b\in Q$ which lie in a common orbit. A path from $a$
to $b$ is, intuitively, a sequence of elements of $Q$, say $%
a_{1},a_{2},...,a_{k}$ such that $b=(...((a\diamond a_{1})\diamond
a_{2})...)\diamond a_{k}$. Since there are $k$ elements in the path, we call
such a path to be a $k$-path. The best way to formalize this is by
constructing an action of a group $G$ on $Q$. Our candidate for $G$ is the
free group on $Q$ itself. Specifically, let $Q=\{a_{1},a_{2},...,a_{n}\}$
and consider $F(Q)$, the free group on $Q$. We define an action of $F(Q)$ on 
$Q$ as follows:

A word $w\in F(Q)$, $w=w_{1}w_{2}...w_{k}$ (with $w_{i}\in Q$ ) acts on an
element $a\in Q$ as:%
\begin{equation*}
a^{w}=(...((a\ast w_{1})\ast w_{2})\ast ...)\ast w_{k}
\end{equation*}%
We will identify the action of $w_{k}^{-1}\in F(Q)$ on $a\in Q$ by $%
a^{w_{k}^{-1}}=a\ast ^{-1}w_{k}$. A \textbf{reduced} $k$-path may be defined
as a reduced word of length $k$ in $F(Q)$ with the unit element $e$ of $F(Q)$
taking the place of the $0$-path. Also, if $w=w_{1}w_{2}...w_{k}$ and $%
u=u_{1}u_{2}...u_{l}$, 
\begin{equation*}
a^{wu}=(...((...((a\ast w_{1})\ast w_{2})\ast ...\ast )w_{k})\ast u_{1})\ast
...)\ast u_{l}=((...((a\ast w_{1})\ast w_{2})\ast ...)\ast
w_{k})^{u}=(a^{w})^{u}
\end{equation*}%
so that the above definition is legitimately an action of $F(Q)$ on $Q$.

Moreover, for any two elements $a,b\in Q$ and any word $w$ in $F(Q)$ we have
that $(a\ast b)^{w}=a^{w}\ast b^{w}$ since if $w=w_{1}w_{2}...w_{k}$, we have%
\begin{equation*}
(a\ast b)^{w}=(...((a\ast b)\ast w_{1})\ast w_{2}...)\ast w_{k}
\end{equation*}

\begin{equation*}
=(...(((a\ast w_{1})\ast (b\ast w_{1}))\ast w_{2})...)\ast w_{k}
\end{equation*}

\begin{equation*}
=(...(((a\ast w_{1})\ast w_{2})\ast (b\ast w_{1})\ast w_{2}))\ast
w_{3})...)\ast w_{k}
\end{equation*}

\begin{equation*}
=((...((a\ast w_{1})\ast w_{2})...)\ast w_{k})\ast ((...((b\ast w_{1})\ast
w_{2})...)\ast w_{k})=a^{w}\ast b^{w}
\end{equation*}

We also define a path $w=w_{1}w_{2}...w_{k}$ to be disjoint from another
path $u=u_{1}u_{2}...u_{l}$ if $w_{k}u_{1}\neq e$ and $u_{l}w_{1}\neq e$.

\textbf{Example 1:} Consider the (commutative) quandle $Q=%
\{x_{1},x_{2},x_{3}\}$ with the relationships:%
\begin{equation*}
x_{i}\ast x_{j}=x_{k},i\neq j\neq k
\end{equation*}%
and%
\begin{equation*}
x_{i}\ast x_{i}=x_{i}
\end{equation*}%
for $i,j,k=1,2,3$. This is simply the dihedral quandle $D_{3}$. The word $%
w=x_{1}x_{2}x_{3}$ is a $3$-path from $x_{1}$ to $x_{3}$, and is also a path
from $x_{2}$ back to $x_{2}$ (i.e. it is a loop at $x_{2}$\ ).

\section{The orbit group $O(Q,a)$}

In this section, we define the orbit group of a quandle at a point.
Henceforth, we will use the term "point" to refer to an element of $Q$.
Given a point $a\in Q$, a path starting and ending at $a$ is a \textbf{loop}
at $a$. The set of all loops is just the stabilizer $G_{a}$ of the action of 
$F(Q)$ on $Q$. Also, if $a,b$ are connected points of $Q$ (connected by a
path $\alpha $ disjoint from all loops at $a$ and $b$), then $G_{a}$ is
isomorphic to $G_{b}$ via the (invertible) map $p\longmapsto \alpha
^{-1}p\alpha $ that sends loops at $a$ to loops at $b$. Its inverse is the
map $q\longmapsto \alpha q\alpha ^{-1}$ that sends loops at $b$ to loops at $%
a$. The map is clearly a homomorphism since $\alpha ^{-1}pq\alpha =(\alpha
^{-1}p\alpha )(\alpha ^{-1}q\alpha )$.

A simple question to be asked is what sub-quandles of $Q$ are invariant
under the action of $G_{a}$ on $Q$. To answer this, we need the following
definition:

\begin{definition}
Let $a\in Q$. We define the kernel $ker(a)$ of $a$ as follows:%
\begin{equation*}
ker(a)=\{b\in Q|\text{ }b\ast a=a\}
\end{equation*}
\end{definition}

We now have the following:

\begin{proposition}
The kernel $ker(a)$ is invariant under the action of $G_{a}$, in the sense
that for each $p\in G_{a}$ and $b\in ker(a)$ we have that $b^{p}\in ker(a)$.
\end{proposition}

\begin{proof}%
Let $b\in ker(a)$ and consider the product $b^{p}\ast a$. We need to show
that $b^{p}\ast a=a$, since then $b^{p}\in ker(a)$ by the definition above.
To see this, we have the following string of equalities:%
\begin{equation*}
b^{p}\ast a=b^{p}\ast a^{p}=(b\ast a)^{p}=a^{p}=a
\end{equation*}%
Thus, $b^{p}\in ker(a)$.%
\end{proof}%
\medskip

In all that follows, loops are assumed to be reduced. We will also denote by 
$|x|$, to mean $x$ as considered as an element of $Q$. For example, if $q\in
Q$ then $a^{q^{-1}}=a\ast ^{-1}q$ for all $a\in Q$. However, $a^{|q|}$ is,
by definition, $a\ast q$. Thus, the map $|$ $\cdot $ $|$ "forgets" the
inverse action of $Q$ on itself. With this in mind, we have:

\begin{definition}
Two loops $p=w_{1}w_{2}...w_{k}$ and $q=u_{1}u_{2}...u_{l}$ with $p,q\in
F(Q)-\{e\}$ are \textbf{homotopic} if for all $w_{i}$, $i=1,2,...,k$ 
\begin{equation*}
|w_{i}|\text{ }\in O_{|u_{j}|}
\end{equation*}%
for some $j=1,2,...,l$ and for all $u_{j}$, 
\begin{equation*}
|u_{j}|\text{ }\in O_{|w_{i}|}
\end{equation*}%
for some $i=1,2,...,k$. The only loop homotopic to $e\in F(Q)$ is defined to
be itself.
\end{definition}

\textbf{Example 1:} Consider $D_{4}$, the dihedral quandle on $%
\mathbb{Z}
/4%
\mathbb{Z}
$. The orbit decomposition of this quandle is 
\begin{equation*}
D_{4}=O_{1}\sqcup O_{2}
\end{equation*}%
with $\{1,3\}=O_{1}$ and $\{0,2\}=O_{2}$. $p=121$ and $q=13$ are loops at $%
0\in Q$, as 
\begin{equation*}
((0\ast 1)\ast 2)\ast 1=(2\ast 2)\ast 1=(2\ast 1)=0
\end{equation*}%
and 
\begin{equation*}
(0\ast 1)\ast 3=2\ast 3=0
\end{equation*}%
But $p$ and $q$ are not homotopic loops since the letter $2$ of $p$ is not
in the orbit of $1$ or $3$, although all letters of $q$ are in the orbits of
some letters of $p$. The loops $r=00$ and $s=22$ are, however, homotopic.

\begin{proposition}
The relation $\sim $ on $G_{a}$ defined by $p\sim q$ if $p$ and $q$ are
homotopic is an equivalence relation.
\end{proposition}

\begin{proof}
Reflexivity is evident. For symmetry, if $p\sim q$ then each letter of $p$
is connected to some letter of $q$ and vice versa hence $q\sim p$.
Transitivity is also verified in a similar manner%
\end{proof}%
\medskip

To complete our construction, we will need to define a product on $%
G_{a}/\sim $. We do this as follows:

Given any two classes of $G_{a}$, say $[p]$ and $[q]$, define 
\begin{equation*}
\lbrack p]\ast \lbrack q]=[pq]
\end{equation*}%
with the product $pq$ obtained from $G_{a}$. Not that this product $\ast $
is not related to the quandle product.

The product $\ast $ is well defined. To see this, let $p_{1}\sim p_{2}$ and $%
q_{1}\sim q_{2}$ with $p_{1}=a_{1}a_{2}...a_{m}$, $p_{2}=b_{1}b_{2}...b_{n}$%
, $q_{1}=c_{1}c_{2}...c_{k}$ and $q_{2}=d_{1}d_{2}...d_{l}$ then 
\begin{equation*}
p_{1}q_{1}=a_{1}a_{2}...a_{m}c_{1}c_{2}...c_{k}
\end{equation*}%
Each letter of $p_{1}q_{1}$ is connected to some letter of $p_{2}q_{2}$ and
vice versa since every $a_{i}$ is connected to some $b_{j}$, (and vice
versa). The same is true for the $c$ and $d$ subscript terms which shows
that $p_{1}q_{1}\sim p_{2}q_{2}$. We define the orbit group $O(Q,a)$to be $%
G_{a}/\sim $ with the product as defined above, i.e. $[p]\ast \lbrack
q]=[pq] $

The last step is to verify that $O(Q,a)$ is a group:

\begin{proposition}
$O(Q,a)$ is a group with $\ast $ as the product
\end{proposition}

\begin{proof}%
The product $\ast $ is closed and associative from the operation in $G_{a}$.
Given a class $[p]$, the inverse class $[p]^{-1}$ is simply $[p^{-1}]$ since 
$[p]\ast \lbrack p^{-1}]=[p^{-1}]\ast \lbrack p]=[1]$ and $[p^{-1}]$ is
unique since $p^{-1}$ is unique. 
\end{proof}%
\medskip

Before we proceed, the following lemma will be necessary for some of our
results:

\begin{proposition}
$O(Q,a)=%
\mathbb{Z}
/2%
\mathbb{Z}
$ for connected quandles $Q$.\footnote{%
If $Q$ is connected, $O(Q,a)=O(Q,b)=%
\mathbb{Z}
/2%
\mathbb{Z}
$ for all $a,b$ in $Q$, so that we can write $O(Q)=%
\mathbb{Z}
/2%
\mathbb{Z}
$ without ambiguity}
\end{proposition}

\begin{proof}%
Let $p=a_{1}a_{2}...a_{m}$ be a loop at $a$, with $|a_{i}|$ $\in Q$. Since $%
Q $ is connected each $a_{i}\in O_{a}$ so that $[p]=[a]$. Likewise, each
non-trivial loop at $a$ is homotopic to $[a]$, so that $O(Q,a)$ consists of
only $[1]$ and $[a].$ Also, $[a]=[a^{-1}]$ since $|a^{-1}|$ $=a$ and $a$ is
connected to itself. Hence is $O(Q,a)$ isomorphic to $%
\mathbb{Z}
/2%
\mathbb{Z}
$%
\end{proof}%
\medskip

As our first example, we compute the orbit group of $T_{n}$, the trivial
quandle of $n$ elements. In this case, there are $n$ orbits consisting of
only single elements i.e. $O_{a}=\{a\}$ and every loop is homotopic to its
inverse so that $[p]\ast \lbrack p]=[1]$. Also, $G_{a}=F(Q)$ for each $a$ in 
$Q$. Thus, given any word $p$ in $F(Q)$, $p=a_{1}a_{2}...a_{m}$ we have that 
$[p]=[a_{1}]\ast \lbrack a_{2}]...\ast \lbrack a_{m}]$ since each class is
not homotopic to any other class except itself. So, we can conclude that 
\begin{equation*}
O(T_{n},a)=O(T_{n})=%
\mathbb{Z}
/2%
\mathbb{Z}
\oplus 
\mathbb{Z}
/2%
\mathbb{Z}
\oplus ...\oplus 
\mathbb{Z}
/2%
\mathbb{Z}
\text{ }(n\text{ factors})
\end{equation*}%
Note that the group is independent of the base point, showing that
connectivity is not necessary for base point independence. For connected
quandles, $[p]=[p^{-1}]=[p]^{-1}$ so that $[p]\ast \lbrack p]=[1]$ for all
loops in a conected quandle.

Our next example is the dihedral quandle $D_{n}$. Assume first that $n$ is
odd and consider the products $0\ast j=2j$ $(mod$ $n)$ for $j=0,1,...,n-1$.
Thus, $\{0\ast j|$ $j=0,1,...,n-1\}=\{0,2,4,...,n-1,1,3,...,n-2\}$. Clearly
this product results in the generation of $D_{n}$, so that $O_{0}=D_{n}$.
Next, assume $n$ is even and again consider the orbit of $0$. The products $%
0\ast j$ are $0,2,4,...,n$ for $j=0,1,...,n/2$ and these products repeat for 
$j=n/2+1,...n$, so that $O_{0}=\{0,2,...,n-2\}$. Likewise, $%
O_{1}=\{1,3,...n-1\}$ which shows that $D_{n}=O_{0}\sqcup O_{1}$

As $D_{n}$ is connected for odd $n$, we have:

\begin{corollary}
$O(D_{n},i)=%
\mathbb{Z}
/2%
\mathbb{Z}
$ for all $i=0,1,...(n-1)$ and $n$ odd.
\end{corollary}

\begin{proof}%
Trivial from the above proposition%
\end{proof}%
\medskip

To motivate our discussion further, let us compute $O(D_{4},0)$. From the
example above, $D_{4}=O_{1}\sqcup O_{2}$. Let $p=a_{1}a_{2}...a_{m}$ be a
loop at $0$. Since each of the $a_{i}$ are in either $O_{1}$ or $O_{2}$, $%
[p] $ is either $[0]$, $[11]$ or $[0]\ast \lbrack 11]$ for all loops (the
loops $[11]$ and $[0]$ are, as noted above, not homotopic). Also, $[11]\ast
\lbrack 0]=[0]\ast \lbrack 11]$ so that 
\begin{equation*}
O(D_{4},0)=%
\mathbb{Z}
/2%
\mathbb{Z}
\oplus 
\mathbb{Z}
/2%
\mathbb{Z}%
\end{equation*}

The same computation can be extended to $D_{n}$ for even $n$ to obtain:

\begin{proposition}
$O(D_{n},i)=%
\mathbb{Z}
/2%
\mathbb{Z}
\oplus 
\mathbb{Z}
/2%
\mathbb{Z}
$ for all $i=0,1,...(n-1)$ and $n$ even
\end{proposition}

\begin{proof}%
Same computation as above%
\end{proof}%
\medskip

It seems that the number of factors of $%
\mathbb{Z}
/2%
\mathbb{Z}
$ of the orbit group corresponds to the number of orbits of the quandle,
which would show that the orbit group counts the number of orbits of a
quandle. This is true for certain quandles:

\begin{theorem}
If $Q=\sqcup O_{a_{i}}$, $i=1,2,...,k$ with $O_{a_{i}}$ the orbits of $Q$,
such that for each $O_{a_{i}}$, there is an $x_{i}$ in $O_{a_{i}}$ with $%
a\ast x_{i}=a$, then 
\begin{equation*}
O(Q,a)=G_{k}=%
\mathbb{Z}
/2%
\mathbb{Z}
\oplus 
\mathbb{Z}
/2%
\mathbb{Z}
\oplus ...\oplus 
\mathbb{Z}
/2%
\mathbb{Z}
\text{ }(k\text{ factors})
\end{equation*}
\end{theorem}

\begin{proof}%
Since the $O_{a_{i}}$ are disjoint, the $a_{i}$ are not connected and also $%
[x_{i}x_{j}]=[x_{i}]\ast \lbrack x_{j}]=[x_{j}]\ast \lbrack x_{i}]$, as $%
x_{i}\in O_{a_{i}}$. Let $p$ be a loop at $a$ so that $a^{p}=a$. Thus $%
p=w_{1}w_{2}...w_{m}$ with $|w_{j}|$ in one of the $O_{a_{i}}$. Now, $%
[p]=[w_{1}w_{2}...w_{m}]=[x_{1}]\ast \lbrack x_{2}]\ast ...\ast \lbrack
x_{l}]$, $l\leq k$. As each of $[x_{i}]$ commute with $[x_{i}]=[x_{i}]^{-1}$%
, we have that they are in $G_{k}$ and hence each $[p]$ is also an element
of $G_{k}$. Since there is an $x_{i}$ in each $O_{a_{i}}$ with $a\ast
x_{i}=a $, the $[x_{i}]$ generate the entire group.%
\end{proof}%

\end{document}